\documentclass{amsart}
\usepackage{hyperref}
\usepackage{texdraw}

\usepackage{tikz}
\usetikzlibrary{matrix,decorations.pathreplacing}
\usepackage{pgfplots}

\makeatletter

\newfont{\gothic}{eufm10 scaled 1100}

\theoremstyle{plain}    
\newtheorem{thm}{Theorem}[section]
\numberwithin{figure}{section} 
\theoremstyle{plain}    
\newtheorem{cor}[thm]{Corollary} 
\theoremstyle{plain}    
\newtheorem{cor-def}[thm]{Corollary-Definition} 
\theoremstyle{plain}    
\newtheorem{conj}[thm]{Conjecture} 
\theoremstyle{plain}    
\theoremstyle{plain}
\newtheorem{lem}[thm]{Lemma} 
\theoremstyle{plain}    
\newtheorem{prop}[thm]{Proposition} 
\theoremstyle{plain}    
\newtheorem{Def}[thm]{Definition} 
\theoremstyle{remark}

\theoremstyle{remark}

\usepackage[arrow,matrix,curve,ps]{xy}
\xyoption{dvips}
\CompileMatrices

\makeatother

\begin{document}

\title{Iterative dissection of Okounkov bodies of graded linear series on $\mathbb{CP}^2$}


\author{Thomas Eckl}

\keywords{Okounkov bodies, Nagata Conjecture}

\subjclass{14J26, 14C20}


\address{Thomas Eckl, Department of Mathematical Sciences, The University of Liverpool, Mathematical
               Sciences Building, Liverpool, L69 7ZL, England, U.K.}

\email{thomas.eckl@liv.ac.uk}

\urladdr{http://pcwww.liv.ac.uk/~eckl/}

\maketitle

\begin{abstract}
Let $\pi: X \rightarrow \mathbb{P}^2$ be the blow-up of $\mathbb{CP}^2$ in $n$ points $x_i$ in very general position, and let $E_i$ be the exceptional divisor over $x_i$.
For $0 \leq n \leq 9$ we calculate Okounkov bodies of graded linear series given by sections of multiples of line bundles $\pi^\ast \mathcal{O}_{\mathbb{P}^2}(d) \otimes \mathcal{O}_X(-m\sum_{i=1}^n E_i)$ with respect to a flag consisting of a line on $\mathbb{CP}^2$ and a point on the line in general position. Furthermore, we show what Nagata's Conjecture predicts on these Okounkov bodies when $n > 9$.
\end{abstract}


\pagestyle{myheadings}
\markboth{THOMAS ECKL}{ITERATIVE DISSECTION OF OKOUNKOV BODIES}

\setcounter{section}{-1}

\section{Introduction}

\noindent In 1996, Okounkov (\cite{Oko96, Oko03}) constructed in a side remark convex polytopes associated to graded linear series on projective algebraic varieties, generalizing the construction of Newton polytopes associated to single polynomials and of moment polytopes associated to toric varieties. As was made precise later on by the work of Lazarsfeld and Musta\c{t}\u{a} \cite{LM08} and Khovanskii and Kaveh \cite{KK08}, these \textit{\mbox{(Newton-)}Okounkov bodies} encode invariants of the graded linear series. In recent years there were also striking applications of this connection between Algebraic Geometry and Convex Geometry to the study of semigroups \cite{KK12} and to Symplectic Geometry \cite{HK12}.

\noindent However, Okounkov bodies are still difficult to explicitely calculate whenever we leave the toric situation, even in simple settings on projective complex algebraic surfaces (see e.g the challenge stated in \cite{DKMS13}). This is not too surprising, as the geometry of Okounkov bodies is closely related to the structure of the cone of big divisors on a projective algebraic variety, and this cone is known to have an intricate structure for a long time. For example, it is still not known what the big cone of $\mathbb{CP}^2$ blown up in $n>9$ points looks like (see \cite[Ch.5]{LazPAG1} for some conjectures on the structure of the dual ample cone). 

\noindent On the other hand, Lazarfeld's and Musta\c{t}\u{a}'s characterisation of Okounkov bodies associated to complete linear series on projective complex algebraic surfaces (\cite[Thm.6.4]{LM08}, completed by \cite[Thm.B]{KLM13}) contains an algorithm to calculate these Okounkov bodies when we have enough knowledge of the big cone of this surface. {\L}uszcz-{\'S}widecka and Schmitz \cite{LS14} introduced further shortcuts to make the algorithm more efficient. 

\noindent In this note, we investigate the Okounkov bodies associated to certain complete linear series on $\mathbb{CP}^2$ blown up in $n$ points in general position. We are able to calculate the Okounkov bodies if $n \leq 9$ since the big cones of $\mathbb{CP}^2$ blown up in at most $9$ points are well enough understood. The new technical ingredient is the combination of the algorithms of Lazarsfeld and  Musta\c{t}\u{a}, with the amendments in \cite{KLM13}, and the symmetry of the considered linear systems. As a surprising corollary we are able to predict the form of the Okounkov bodies for $n > 9$ using Nagata's Conjecture on the big cone of $\mathbb{CP}^2$ blown up in more than $9$ points: The genericity of their form corresponds to the non-special behaviour of the linear systems predicted by the Conjecture.  For exact statements and proofs, see section~\ref{calc-sec}.

\section{Okounkov bodies on surfaces and Zariski decompositions} \label{OB-ZD-sec}

\noindent Assume from now on that all algebraic varieties are complex. 

\noindent To prove the results in section~\ref{calc-sec}, we need a characterisation of Okounkov bodies of big $\mathbb{R}$-divisors on smooth projective complex algebraic surfaces, first stated by Lazarsfeld and Musta\c{t}\u{a} \cite[Thm.6.4]{LM08} and completed by K\"uronya, Lozovanu and Maclean \cite[Thm.B]{KLM13}:
\begin{thm} \label{OB-surf-thm}
Let $D$ be a big $\mathbb{R}$-divisor on a smooth projective algebraic surface $X$, and let $Y_\bullet: X \supset C \supset \{x\}$ be an admissible flag on $X$, with $C$ an irreducible and reduced curve on $X$ and $x \in C$ a nonsingular point on $C$. Set
\[ \mu = \mu(D;C) := \sup \{s > 0 | D - sC\ \mathrm{is\ big} \}. \]
Then there exist continuous functions $\alpha, \beta: [a,\mu] \rightarrow \mathbb{R}_+$ for some $0 \leq a \leq \mu$ with $\alpha$ convex and increasing, $\beta$ concave, $\alpha \leq \beta$, and both $\alpha$ and $\beta$ piecewise linear with rational slopes and only finitely many breakpoints such that the Okounkov body $\Delta_{Y_\bullet}(D) \subset \mathbb{R}^2_+$ is the region bounded by the graphs of $\alpha$ and $\beta$,
\[ \Delta_{Y_\bullet}(D) = \{(t,y) \in \mathbb{R}^2_+ | a \leq t \leq \mu, \alpha(t) \leq y \leq \beta(t) \}. \] 
\end{thm}
\noindent The proof of this characterisation relies on the notion of Zariski decompositions of pseudo-effective $\mathbb{R}$-divisors.
\begin{Def} \label{ZD-def}
Let $D$ be a pseudo-effective $\mathbb{R}$-divisor on a smooth projective algebraic surface $X$. A decomposition $D = P + N$ into a nef $\mathbb{R}$-divisor $P$ (the positive part) and an effective $\mathbb{R}$-divisor $N$ (the negative part) is called a Zariski decomposition of $D$ if $P \cdot C_i = 0$ for all $i = 1, \ldots, q$ and the intersection matrix $(C_i \cdot C_j)_{1 \leq i,j \leq q}$ is negative-definite, where $C_1, \ldots, C_q$ are the reduced and irreducible components of the support of $N$.
\end{Def}
\noindent Existence and uniqueness of Zariski decompositions is a a classical result about divisors on algebraic surfaces, proven e.g. in \cite[Thm.14.14]{Bad01} for $\mathbb{Q}$-divisors, but the proof generalizes to $\mathbb{R}$-divisors.

\noindent Using Zariski decompositions the proof of Thm.~\ref{OB-surf-thm} in \cite{LM08} describes the constant $a$ and the functions $\alpha, \beta$ appearing in the statement in more details:
\begin{itemize}
\item[(i)] $a$ is the coefficient of $C$ in the support of the negative part $N$ in the Zariski decomposition $D = P+N$.
\item[(ii)] If $t \geq a$ and $D_t := D - t \cdot C = P_t + N_t$ is a Zariski decomposition then
\[ \alpha(t) = \mathrm{ord}_x N_{t|C}\ \mathrm{and\ } \beta(t) = \alpha(t) + C \cdot P_t. \]
\end{itemize}

\noindent The breakpoints of the piecewise linear functions $\alpha(t), \beta(t)$ appear when the Zariski decomposition of $D - t \cdot C$ considerably changes. This happens when the line $\{D - t \cdot C | t \geq a\}$ crosses the border between two \textit{Zariski chambers} of the big cone: If $D = P_D + N_D$ is the Zariski decomposition of $D$ define
\[ \mathrm{Neg}(D) := \{ C \subset X\ \mathrm{reduced\ and\ irreducible\ curve}| C \subset \mathrm{Supp}(N_D) \}. \]
Furthermore, if $P$ is nef, set
\[ \mathrm{Null}(P) := \{  C \subset X\ \mathrm{reduced\ and\ irreducible\ curve}| C \cdot P = 0 \}\]
 and 
\[ \mathrm{Face}(P) := \mathrm{Null}(P)^\perp \cap \mathrm{Nef}(X). \]
Then the Zariski chamber of $P$ is defined as
\[ \Sigma_P := \{ D \in \mathrm{Big}(X) | \mathrm{Neg}(D) = \mathrm{Null}(P) \}, \]
and all the Zariski chambers cover the big cone (see \cite[{\S} 1]{BKS04} for more details).

\noindent Using the Hodge Index Theorem \cite[Thm.V.1.9]{Hart:AG} it is easy to show that a non-empty set $\mathrm{Null}(P)$ consists of finitely many curves with negative-definite intersection matrix. We also need a characterisation of the closure of the Zariski chamber $\Sigma_P$:
\begin{lem}[{\cite[Prop.1.8]{BKS04}}] \label{clos-ZC-lem}
$\overline{\Sigma_P}$ is the cone generated by $\mathrm{Face}(P)$ and the curves in $\mathrm{Null}(P)$.
\end{lem}

\begin{cor} \label{Null-cor}
If $P$ is a nef and big $\mathbb{R}$-divisor, but not ample, then $\mathrm{Null}(P) \neq \emptyset$.
\end{cor}
\begin{proof}
If $P$ is nef and big but not ample, $P$ lies on the boundary of the nef cone but still in the interior of the big cone. The nef cone is the Zariski chamber of any divisor $Q$ with $\mathrm{Null}(Q) = \emptyset$. Since all the Zariski chambers cover the big cone there must be another nef divisor $Q^\prime$ with $\mathrm{Null}(Q^\prime) \neq \emptyset$ such that $P \in \overline{\Sigma_{Q^\prime}}$. By Lem.~\ref{clos-ZC-lem} we can decompose $P$ as 
\[ P = \overline{P} + \sum_{C \in \mathrm{Null}(Q^\prime)} a_C C, \]
with $\overline{P} \in \mathrm{Face}(Q^\prime)$. Since the intersection matrix of the finitely many curves in $\mathrm{Null}(Q^\prime)$ is negative-definite and $\overline{P} \cdot C = 0$ for all $C \in \mathrm{Null}(Q^\prime)$, this is a Zariski decomposition of $P$. Since $P$ is nef, $P = \overline{P}$. Since $\overline{P} \in \mathrm{Face}(Q^\prime)$ we have
\[ \emptyset \neq \mathrm{Null}(Q^\prime) \subset \mathrm{Null}(\overline{P}) = \mathrm{Null}(P). \]
\end{proof}

\noindent Finally, by definition the negative part of the Zariski decomposition varies linearly in each Zariski chamber of the big cone. Consequently, breakpoints of the piecewise linear functions $\alpha(t), \beta(t)$ in Thm.~\ref{OB-surf-thm} can only occur for parameters $t$ where the line $\{ D - t \cdot C\}$ crosses the border of a Zariski chamber $\Sigma_P$ (see \cite[Prop.2.1]{KLM13} for further information).

\noindent Thus Thm.~\ref{OB-surf-thm} gives rise to an algorithm how to compute Okounkov bodies from good enough knowledge of the big cone and its decomposition into Zariski chambers. Note that the path given by the line $\{d - t \cdot C\}$ can intersect the Zariski chambers of the big cone in a rather complicated way. {\L}uszcz-{\'S}widecka and Schmitz's algorithm \cite{LS14} simplifies the path used for calculation considerably, but we do not need these improvements in our situation.

\section{Curves on $\mathbb{P}^2$ blown up in several points in general position} \label{curves-sec}

\noindent In this section we collect classically known facts on curves on $\mathbb{P}^2$ and on $\mathbb{P}^2$ blown up in several points needed later on. For some of the proofs we refer to the literature whereas we present others, for lack of clear reference, but without claiming any originality. However we try to explain the use of (very) general position in the arguments in more details than usual.

\noindent Let us first fix some notation. $\pi_n: X_n := X_n(x_1, \ldots, x_n) \rightarrow \mathbb{P}^2$ is supposed to be the blow up of $\mathbb{P}^2$ in $n$ points $x_1, \ldots, x_n \in \mathbb{P}^2$ in general position and $E_i = \pi_n^{-1}(x_i)$ the exceptional divisor over $x_i$. If $L \subset \mathbb{P}^2$ denotes a line, then the Picard classes $e_0, e_1, \ldots, e_n$ of the line bundles
\[ \mathcal{O}_{X_n}(\pi_n^\ast L), \mathcal{O}_{X_n}(E_1), \ldots, \mathcal{O}_{X_n}(E_n) \]
generate the Picard group $A^1(X_n) = \mathbb{Z} \cdot e_0 + \sum_{i=1}^n \mathbb{Z} \cdot e_i$. The canonical line bundle represents the class $k := -3e_0 + e_1 + \cdots + e_n$. 
\begin{Def}
An automorphism $\sigma$ of the free abelian group $A^1(X_n)$ is called a Cremona isometry if the following properties are satisfied:
\begin{itemize}
\item[(i)] $\sigma$ preserves the intersection form on $A^1(X_n)$.
\item[(ii)] $\sigma$ leaves the canonical class $k$ of $X_n$ fixed.
\item[(iii)] $\sigma$ leaves the semigroup of effective classes invariant.
\end{itemize}
\end{Def}

\noindent The group of Cremona isometries on $X_n$ will be denoted by $\mathrm{Cris}(X_n)$. 

\begin{Def}
The group $W_n$ of automorphisms of $A^1(X_n)$ generated by the simple reflections $s_1, \ldots, s_{n-1}, s_n$ given by
\[ s_i(e_i) = e_{i+1}, s_i(e_{i+1}) = e_i\ \mathrm{and\ } s_i(e_j) = e_j\ \mathrm{for\ } i = 1, \ldots, n-1, j \neq i, i+1, \]
\[ s_n(e_0) = 2e_0 - e_1 - e_2 - e_3, s_n(e_1) = e_0 - e_2 - e_3, s_n(e_2) = e_0 - e_1 - e_3, s_n(e_3) = e_0 - e_1 - e_2 \] 
and $s_n(e_j) = e_j$ for $j = 4, \ldots, n$ is called the Weyl group of $X_n$.
\end{Def}
\begin{prop}[{\cite[Thm.1, p.286]{Dol83}}] \label{W-C-thm}
$W_n \subset \mathrm{Cris}(X_n)$. 

\noindent Furthermore, $s \in W_n$ maps $e_0, e_1, \ldots, e_n$ to classes $e_0^\prime, e_1^\prime, \ldots, e_n^\prime$ represented by pullback of lines and exceptional divisors coming from another sequence of blow ups in points $y_1, \ldots, y_n \in \mathbb{P}^2$. In that way, every $s \in W_n$ induces a birational map of the $n$-fold product $(\mathbb{P}^2)^n$ of $\mathbb{P}^2$ onto itself, by setting $s(x_1, \ldots, x_n) := (y_1, \ldots, y_n)$.
\end{prop}

\noindent The proposition implies the following principle: If $C$ is a curve of class $[C]$ on $X(x_1, \ldots, x_n)$ then for every $s \in W_n$ there exists an isomorphic curve on $X(s(x_1, \ldots, x_n))$ of class $s([C])$. Consequently, given a flat family of curves $C \subset X_n(x_1, \ldots, x_n)$ such that the $x_1, \ldots, x_n$ vary in an open subset of $(\mathbb{P}^2)^n$ then for a very general $(x_1, \ldots, x_n) \in (\mathbb{P}^2)^n$ there exists for every of the countably many $s \in W_n$ a curve $C^\prime \subset X_n(x_1, \ldots, x_n)$ isomorphic to a curve $C$ of the family and of class $s([C])$.

\noindent All that can be used to classify the exceptional curves of the first kind on $X_n$, that is all nonsingular rational curves $C \subset X_n$ with $C^2 = -1$:
\begin{thm}[{\cite[Cor.1, p.288]{Dol83}}] \label{self--1-thm}
There is a bijection between the set of exceptional curves of the first kind on $X_n$ and the orbit $W_n e_n$.
\end{thm}

\noindent The main tool to prove this theorem is Noether's Inequality:
\begin{lem}[{\cite[p.288]{Dol83}}]
Let $C$ be an irreducible curve of degree $d$ on $\mathbb{P}^2$ passing through points $x_1, \ldots, x_n$ with multiplicities $m_1 \geq  \cdots \geq m_n$, $n \geq 3$. Assume that $m_2 > 0$ and that the strict transform $\overline{C}$ of $C$ on the blow up $X_n$ of $\mathbb{P}^2$ in $x_1, \ldots, x_n$ is a nonsingular rational curve with $-2 \leq \overline{C}^2 \leq 1$. Then:
\[ d < m_1 + m_2 + m_3. \] 
\end{lem}

\noindent Noether's Inequality can also be used to classify the classes of nonsingular rational curves $C \subset X_n$ with $C^2 = 0$:
\begin{prop} \label{self-0-prop}
There is a bijection between the classes of nonsingular rational curves $C$ with $C^2 = 0$ and the orbit $W_n (e_0-e_1)$.
\end{prop}
\begin{proof}
The strict transform on $X_n$ of a line $L \subset \mathbb{P}^2$ running through $x_1$ but not through $x_2, \ldots, x_n$ represents the class $e_0 - e_1$. The principle above implies that for all $s \in W_n$, $s(e_0-e_1)$ is also represented by a nonsingular rational curve with self-intersection $0$. Hence $W_n (e_0-e_1)$ is injected into the set of all classes of such curves.

\noindent Vice versa, let $C$ be a nonsingular rational curve on $X_n(x_1, \ldots, x_n)$ with self-intersection $0$ and let $de_0-m_1e_1 - \cdots - m_n e_n$ be its class. Since $(x_1, \ldots, x_n)$ is assumed to be very general there is a whole family of such curves in $X_n(x_1^\prime, \ldots, x_n^\prime)$ where $(x_1^\prime, \ldots, x_n^\prime)$ varies in a Zariski-open subset of $(\mathbb{P}^2)^n$. Hence we can use the principle above, and iteratively applying the simple reflections $s_1, \ldots, s_{n-1} \in W_n$ we obtain a nonsingular rational curve on $X_n(x_1, \ldots, x_n)$  with self-intersection $0$ representing $de_0-m_1e_1 - \cdots - m_{n^\prime} e_{n^\prime}$ with $n^\prime \leq n$ and $m_1 \geq \ldots \geq m_{n^\prime} > 0$. We distinguish three cases:
\begin{itemize}
\item $n^\prime=1$: Then $(de_0 - m_1e_1)^2 = d^2 - m_1^2 = 0$ only if $d = m_1$. Since a nonsingular curve only vanishes with multiplicity $1$ in a point, $C$ must be of class $e_0 - e_1$.
\item $n^\prime = 2$: $C$ cannot be the strict transform of the line $L$ through $x_1$ and $x_2$ since $(e_0 - e_1 -e_2)^2 = -1$. Then $0 \leq \pi_n(C) \cdot L = d - m_1 - m_2$ implies $m_1 + m_2 \leq d$, hence $m_1^2 + m_2^2 < (m_1 + m_2)^2 \leq d^2$. Consequently, $(de_0 - m_1e_1 - m_2e_2)^2 = 0$ is impossible.  
\item $n^\prime \geq 3$: We can apply Noether's Inequality and conclude $d < m_1 + m_2 + m_3$. Thus applying the simple reflection $s_n$ (and the principle above) yields a curve whose class has coefficient $2d - m_1 - m_2 - m_3 < d$ for $e_0$. The claim follows by induction.
\end{itemize} 
\end{proof}

\noindent We are ready to prove the characterisation of the cone $\mathrm{Big}(X_n)$ (also denoted by $\overline{NE}(X_n)$ in the literature) for $0 \leq n \leq 8$, using Mori theory. The main point is that all these surfaces are del Pezzo (or Fano), that is, $-K_{X_n}$ is ample for $0 \leq n \leq 8$ (\cite[Thm.24.4]{M74}).
\begin{thm} \label{Big-Xn-thm}
If $n=0$, that is $X_0 = \mathbb{P}^2$, then $\mathrm{Big}(X_0) = \mathbb{R}_+ \cdot [L]$.

\noindent If $n=1$ then $\mathrm{Big}(X_1) = \mathbb{R}_+ \cdot [E_1] + \mathbb{R}_+ \cdot ([\pi_1^\ast L] - [E_1])$.

\noindent If $n \geq 2$ then $\mathrm{Big}(X_n)$ is generated (as a convex cone) by the finitely many classes of exceptional curves of the first kind on $X_n$. 

\noindent For $1 \leq n \leq 8$ the classes of exceptional curves of the first kind are given by 
\[ \begin{array}{c|cccccccc|c}
   d & m_1 & m_2 & m_3 & m_4 & m_5 & m_6 & m_7 & m_8 & \# (n=8) \\ \hline
   0 & -1 & & & & & & & & 8 \\
   1 & 1 & 1 & & & & & & & 28 \\
   2 & 1 & 1 & 1 & 1 & 1 & & & & 56 \\
   3 & 2 & 1 & 1 & 1 & 1 & 1 & 1 & & 56 \\
   4 & 2 & 2 & 2 & 1 & 1 & 1 & 1 & 1 & 56 \\
   5 & 2 & 2 & 2 & 2 & 2 & 2 & 1 & 1 & 28 \\
   6 & 3 & 2 & 2 & 2 & 2 & 2 & 2 & 2 & 8
    \end{array} \] 
This table should be read as follows: For fixed $n$ between $1$ and $8$ only those rows in the table matter where $m_{n+1} = \cdots = m_8 = 0$. In each row there is exactly one exceptional line for each permutation of $(m_1, \ldots, m_n)$, and the number of all such permutations is recorded in the last column.
\end{thm}
\begin{proof}
The case $n=0$ is obvious. 

\noindent If $n=1$ then $E_1^2 = -1, E_1 \cdot K_{X_1} = -1$ show that $[E_1]$ generates an extremal ray, by the classification of extremal rays (\cite[Thm.1.4.8]{Mat02}). Furthermore the linear system $|\pi_1^\ast L - E_1|$ induces the projection of $X_1$ to $\mathbb{P}^1$, with fibers $\cong \mathbb{P}^1$, and $(\pi_1^\ast L - E_1) \cdot K_{X_1} = -2$ shows that $[\pi_1^\ast L - E_1]$ generates an extremal ray, again by the classification of extremal rays. Since $-K_{X_1}$ is ample, the Cone Theorem \cite[Thm.1.3.1]{Mat02} implies that $\mathrm{Big}(X_1)$ is generated (as a convex cone) by finitely many extremal rays. Since $\dim_{\mathbb{R}} A^1(X_1)_{\mathbb{R}} = 2$ there cannot be more than $2$ extremal rays in this case.

\noindent If $2 \leq n \leq 8$ nonsingular rational curves of self-intersection $0$ on $X_n$ represent classes in $W_n(e_0-e_1)$, by Prop.~\ref{self-0-prop}. But $e_0 - e_1 = (e_0 - e_1 - e_2) + e_2$ is a sum of exceptional curves of the first kind, hence this class does not generate an extremal ray of $\mathrm{Big}(X)$. Since exceptional curves of the first kind represent all classes in $W_n e_1$ by Thm.~\ref{self--1-thm} it follows that a nonsingular rational curve of self-intersection $0$ cannot be an extremal ray of $\mathrm{Big}(X_n)$. Consequently, the only possible extremal rays of $\mathrm{Big}(X_n)$ are exceptional curves of the first kind, by the classification of extremal rays, and since $-K_{X_n}$ is ample, they generate $\mathrm{Big}(X_n)$ by the Cone Theorem.

\noindent The table is obtained by calculating the orbit $W_n e_1$ (see~\cite[p.135]{M74}).
\end{proof}

\noindent This characterisation of the big cone of $X_n$ can be used to calculate multipoint Seshadri constants on $\mathbb{P}^2$:
\begin{cor} \label{Sesh-cor}
Let $\epsilon_n := \epsilon_{\mathbb{P}^2}(L; x_1, \ldots, x_n) := \sup \{ t > 0 | \pi_n^\ast L - t \cdot \sum_{i=1}^n E_i\ \mathrm{is\ ample} \}$ denote the $n$-point Seshadri constant of the divisor $L$ on $\mathbb{P}^2$. For $1 \leq n \leq 9$ it is given by the following table:
\[ \begin{array}{c||c|c|c|c|c|c|c|c|c}
    n & 1 & 2 & 3 & 4 & 5 & 6 & 7 & 8 & 9 \\ \hline
    \epsilon_n & 1 & \frac{1}{2} & \frac{1}{2} & \frac{1}{2} &  \frac{2}{5} & \frac{2}{5} & \frac{3}{8} & \frac{6}{17} & \frac{1}{3} 
    \end{array}\]
\end{cor}
\begin{proof}
As shown in Thm.~\ref{Big-Xn-thm} the convex cone $\mathrm{Big}(X_n)$ is generated by finitely many extremal rays for $1 \leq n \leq 8$. Kleiman's Criterion \cite[Thm.1.4.29]{LazPAG1} tells us that the Seshadri constant $\epsilon_n$ is the maximal number $t$ such that the intersection of $\pi_n^\ast L - t \cdot \sum_{i=1}^n E_i$ with all exceptional curves of the first kind (and $\pi_1^\ast L - E_1$ if $n=1$) is non-negative. This maximum can be read off the table in \ref{Big-Xn-thm}.

\noindent $\epsilon_9 = \frac{1}{3}$ holds because Nagata's Conjecture is true for square numbers \cite[Rem.5.1.14]{LazPAG1}.
\end{proof}

\begin{conj}[Nagata] \label{Nag-conj}
For $n > 9$ we have $\epsilon_n = 1/\sqrt{n}$.
\end{conj}

\section{Calculation of Okounkov bodies} \label{calc-sec}

\noindent Let $L \subset \mathbb{P}^2$ be a line and $y \in L$ a point, let $\pi_n: X_n \rightarrow \mathbb{P}^2$ be the blow up of $\mathbb{P}^2$ in $n$ points $x_1, \ldots, x_n$ in very general position, and let $E_i = \pi_n^{-1}(x_i)$ denote the exceptional line over $x_i$. Then for positive integers $d, m$ we consider divisors of the form
\[ L_{n,d,m} = d\pi_n^\ast L - m \cdot \sum_{i=1}^k E_i \]
and study the Okounkov bodies $\Delta_{Y_\bullet}(L_{n,d,m}) := \Delta_{Y_\bullet}(V_\bullet^{(n,d,m)})$ associated to the complete linear series
\[ V_\bullet^{(n,d,m)} = \{V_k^{(n,d,m)}\}_{k \in \mathbb{N}} = 
                                                                     \{H^0(X_n, \mathcal{O}_{X_n}(kL_{n,d,m}) \}_{k \in \mathbb{N}} \]
with respect to the flag $Y_\bullet: X_n \supset \pi_n^{-1}(L) \supset \{\pi_n^{-1}(y)\}$ (see \cite{LM08} for notation, constructions and proofs).  

\noindent To this purpose we first recall a fact that shows the inclusions
\begin{equation}
\Delta_{Y_\bullet}(V_\bullet^{(n,d,m)}) \subset \Delta_{Y_\bullet}(V_\bullet^{(n^\prime,d,m)}),\ n > n^\prime. 
\end{equation}

\begin{prop}[see~{\cite[Prop.4.1]{DKMS13}}] \label{incl-prop}
Let $X$ be a smooth projective algebraic surface and $\pi: \widetilde{X} \rightarrow X$ the blow up of $X$ in a point $p \in X$. Denote by $E \subset \widetilde{X}$ the exceptional line over $p$, and let $Y \subset X$ be a smooth irreducible curve such that $p \not\in Y$ and $y \in Y$ a point, constituting an admissible flag $Y_\bullet: \widetilde{X} \supset Y \supset \{y\}$ on $X$ that can also be taken as an admissible flag on $\widetilde{X}$. Finally, let $L$ be a divisor on $X$. Then for any $k \in \mathbb{N}$:
\[  \Delta_{Y_\bullet}(\pi^\ast L - k \cdot E) \subset \Delta_{Y_\bullet}(L) \subset \mathbb{R}^2. \]
\end{prop}
\begin{proof}
For all $n \in \mathbb{N}$ there is a natural inclusion
\[ H^0(\widetilde{X}, \mathcal{O}_{\widetilde{X}}(n\pi^\ast L - nk \cdot E)) \hookrightarrow H^0(X, \mathcal{O}_X(nL)) \]
identifying sections of $\mathcal{O}_{\widetilde{X}}(n\pi^\ast L - nk \cdot E)$ with sections of $\mathcal{O}_X(nL)$ having multiplicity $\geq nk$ in $p$. Identified sections are equal when identifying the line bundles  $\mathcal{O}_{\widetilde{X}}(n\pi^\ast L - nk \cdot E)$ and $\mathcal{O}_X(nL)$ on $X - \{p\} \cong \widetilde{X} - E$. Hence the valuations of these sections calculated with respect to the flag $Y_\bullet$ are equal, and the inclusion of Okounkov bodies follows. 
\end{proof}

\noindent Applying this proposition to an iterative sequence of blow ups can be used to define the following notion:
\begin{cor-def} \label{incl-cor}
Let $X$ be a smooth projective algebraic surface, and 
\[ X = X_0 \stackrel{\pi_1}{\leftarrow} X_1 \stackrel{\pi_2}{\leftarrow} \cdots \stackrel{\pi_n}{\leftarrow} X_n = \widetilde{X}\]
a sequence of blow ups in points $p_i \in X_i$, $i = 0, \ldots, n-1$. Let $D_0, \ldots, D_n$ be divisors on $X_0, \ldots, X_n$ such that 
\[ \pi_i^\ast D_{i-1} - D_i\ \mathit{is\ effective\ and}\ (\pi_i)_\ast(D_i) = D_{i-1},\ i = 1, \ldots, n. \]
Let $Y \subset X$ be a smooth irreducible curve and $y \in Y$ a point such that $p_0, (\pi_0 \circ \ldots \circ \pi_i)(p_i) \not\in Y$, $i = 1, \ldots, n-1$. Consider the admissible flag $Y_\bullet: X \supset Y \supset \{y\}$ which also can be seen as an admissible flag on all the $X_i$, $i = 1, \ldots, n$. Then:
\[ \Delta_{Y_\bullet}(D_0) \supset \Delta_{Y_\bullet}(D_1) \supset \cdots \supset \Delta_{Y_\bullet}(D_n). \]
This chain of inclusions is called the iterative Okounkov body dissection associated to $\pi_1, \ldots, \pi_n$ and $D_n$ (and the flag $Y_\bullet$). \hfill $\Box$
\end{cor-def}

\noindent The next theorem shows how the Okounkov bodies $\Delta_{Y_\bullet}(L_{n,m,d})$ with fixed $n$ but varying $d, m$ can be calculated from each other. The main reason for this connection is that the curve $L$ in the flag $Y_\bullet$ is so closely linked to the divisors $L_{n,d,m}$.

\begin{thm}
With notation as above, choose $n, d, m$ such that $L_{n,d,m}$ is big. Then for all $d^\prime, m^\prime$ such that $\epsilon^\prime := \frac{m^\prime}{d^\prime} \geq \frac{m}{d} =: \epsilon$,
\[ \frac{1}{d^\prime} \cdot \Delta_{Y_\bullet}(L_{n,d^\prime,m^\prime}) = \phi_{\epsilon^\prime/\epsilon}(\frac{1}{d} \cdot \Delta_{Y_\bullet}(L_{n,d,m})) \cap \Delta_{Y_\bullet}(\pi_n^\ast L) \]
where
\begin{equation} \label{rescale-form} 
\phi_r: \mathbb{R}^2 \rightarrow \mathbb{R}^2,\ (x,y) \mapsto r \cdot (x-1,y) + (1,0) 
\end{equation}
is the radial rescaling of $\mathbb{R}^2$ by a factor $r$ with center $(1,0)$.
\end{thm}
\begin{proof}
 The proof of Prop.~\ref{incl-prop} and a standard calculation of Okounkov bodies (see \cite[Ex.2.3(a)]{KMS12}) show that
\[ \Delta_{Y_\bullet}(\pi_n^\ast L) = \Delta_{Y_\bullet}(L) = \{ (x,y) \in \mathbb{R}^2 | 0 \leq x, y, x + y \leq 1 \},  \]
the $2$-simplex. By Cor.~\ref{incl-cor} and a simple rescaling property of Okounkov bodies, $\frac{1}{d} \cdot \Delta_{Y_\bullet}(L_{n,d,m})$ and $\frac{1}{d^\prime} \cdot \Delta_{Y_\bullet}(L_{n,d^\prime,m^\prime})$ are both contained in $\Delta_{Y_\bullet}(\pi_n^\ast L)$. To apply Thm.~\ref{OB-surf-thm} to $L_{n,d,m}$ we set
\[ \mu_\epsilon := \sup \{ t > 0 | (1-t) \pi_n^\ast L - \epsilon \cdot \sum_{i=1}^n E_i\ \mathrm{is\ big} \}. \]
Then there exist a constant $0 \leq a_\epsilon \leq \mu_\epsilon$ and functions $\alpha_\epsilon, \beta_\epsilon: [a_\epsilon, \mu_\epsilon] \rightarrow \mathbb{R}_+$ such that
\[ \frac{1}{d} \cdot \Delta_{Y_\bullet}(L_{n,d,m}) = \{ (t,y) \in \mathbb{R}^2_+ | a_\epsilon \leq t \leq \mu_\epsilon, \alpha_\epsilon(t) \leq y \leq \beta_\epsilon(t) \}. \]
As discussed above $a_\epsilon$ is the coefficient of the irreducible divisor $\pi_n^{-1}(L)$ in the negative part $N_\epsilon$ of the Zariski decomposition $\pi_n^\ast L - \epsilon \cdot \sum_{i=1}^m E_i = P_\epsilon + N_\epsilon$. But $\pi_n^{-1}(L)$ cannot appear in the support of $N_\epsilon$ because the intersection matrix of the irreducible components of this support must be negative-definite, by Def.~\ref{ZD-def}. Hence $a_\epsilon = 0$.  

\noindent The functions $\alpha_\epsilon(t), \beta_\epsilon(t)$ are defined on the interval $[0, \mu_\epsilon]$ using the Zariski decompositions
\[ D_{t,\epsilon} := \pi_n^\ast L - \epsilon \cdot \sum_{i=1}^n E_i - t \cdot \pi_n^\ast L = (1 - t) \cdot \pi_n^\ast L - \epsilon \cdot \sum_{i=1}^n E_i = P_{t,\epsilon} + N_{t,\epsilon}, \]
by $\alpha_\epsilon(t) = \mathrm{ord}_x  N_{t,\epsilon|\pi_n^{-1}(L)}$ and $\beta_\epsilon(t) = \alpha_\epsilon(t) +  \pi_n^{-1}(L) \cdot P_{t,\epsilon}$. 

\noindent When we replace $\epsilon$ by $\epsilon^\prime$ the same formulas hold for $a_{\epsilon^\prime}, \alpha_{\epsilon^\prime}, \beta_{\epsilon^\prime}$ as long as $\pi_n^\ast L - \epsilon^\prime \cdot \sum_{i=1}^n E_i$ is pseudoeffective, that is $\mu_{\epsilon^\prime} \geq 0$. Note that $D_{s, \epsilon^\prime} = \epsilon^\prime/\epsilon \cdot D_{t, \epsilon}$ where $s = \epsilon^\prime/\epsilon \cdot (t-1) + 1$. Furthermore, $\epsilon^\prime/\epsilon \cdot D_{t, \epsilon} = \epsilon^\prime/\epsilon \cdot P_{t,\epsilon} + \epsilon^\prime/\epsilon \cdot N_{t,\epsilon}$ is a Zariski decomposition and $\phi_{\epsilon^\prime/\epsilon}(t,y) = (s, \epsilon^\prime/\epsilon \cdot y)$. Finally $\epsilon^\prime \geq \epsilon$ implies that for all $s \in [0, \mu_{\epsilon^\prime}]$ there exists a $t \in [0, \mu_{\epsilon}]$ such that $s = \epsilon^\prime/\epsilon \cdot (t-1) + 1$. The claim follows.

\noindent The claim also remains true if $\pi_n^\ast L - \epsilon^\prime \cdot \sum_{i=1}^n E_i$ is not any longer pseudo-effective: Then $\mu_{\epsilon^\prime} < 0$, and
$\phi_{\epsilon^\prime/\epsilon}(\frac{1}{d} \cdot \Delta_{Y_\bullet}(L_{n,d,m})) \subset \{ (x,y) \in \mathbb{R}^2: x < 0\}$
does not intersect $\Delta_{Y_\bullet}(L)$.
\end{proof}

\noindent Note that the theorem can be applied for every $n$ because $L_{n,d,m} = \pi_n^\ast L - m \cdot \sum_{i=1}^n E_i$ is ample and hence big whenever $\frac{m}{d}$ is small enough: If $\widetilde{C} \subset \widetilde{X}$ is a reduced and irreducible curve such that $C := \pi_n(\widetilde{C})$ is also a curve , then $\widetilde{C}$ is linearly equivalent to the divisor $\deg(C) \cdot \pi_n^\ast L - \sum_{i=1}^n \mathrm{mult}_{x_i}(C) \cdot E_i$. Since $\mathrm{mult}_{x_i}(C) \leq \deg(C)$ the inequality $\frac{m}{d} < \frac{1}{n}$ implies
\[ L_{n,d,m} \cdot \widetilde{C} = d \cdot \deg(C) - m \sum_{i=1}^n \mathrm{mult}_{x_i}(C) > 0. \]
Since also $L_{n,d,m} \cdot E_i = m > 0$ and $L_{n,m,d}^2 > 0$ for $m \ll d$, the Nakai-Moishezon Criterion \cite[Thm.1.2.23]{LazPAG1} implies that $L_{n,d,m}$ is ample.

\noindent By the rescaling property of Okounkov bodies, the analogon to (\ref{rescale-form}) also holds for $\mathbb{R}$-divisors of the form $\pi_n^\ast L - \epsilon \cdot \sum_{i=1}^n E_i$ for arbitrary $\epsilon \in \mathbb{R}$: For $\epsilon^\prime \geq \epsilon$ and $\pi_n^\ast L - \epsilon \cdot \sum_{i=1}^n E_i$ $\mathbb{R}$-big,
\begin{equation} \label{resc-OB2-form}
\Delta_{Y_\bullet}(\pi_n^\ast L - \epsilon^\prime \cdot \sum_{i=1}^n E_i) = \phi_{\epsilon^\prime/\epsilon}(\Delta_{Y_\bullet}(\pi_n^\ast L - \epsilon \cdot \sum_{i=1}^n E_i)) \cap \Delta_{Y_\bullet}(L). 
\end{equation} 

\noindent The next theorem calculates these Okounkov bodies when 
\[ \epsilon := \epsilon_n := \epsilon_{\mathbb{P}^2}(L; x_1, \ldots, x_n) := \sup \{ t > 0 | \pi_n^\ast L - t \cdot \sum_{i=1}^n E_i\ \mathrm{is\ ample} \} \]
is the multi-point Seshadri constant of the ample divisor $L$ on $\mathbb{P}^2$ in the points $x_1, \ldots, x_n \in \mathbb{P}^2$. Note that $\epsilon_{\mathbb{P}^2}(L; x_1, \ldots, x_n)$ is predicted to be irrational if $n > 9$ is not a square number, by Nagata's Conjecture (see Sec.~\ref{curves-sec}).
\begin{thm} \label{Sesh-OB-thm}
Setting $D_n := \pi_n^\ast L - \epsilon_n \cdot \sum_{i=1}^n E_i$, the Okounkov body of $D_n$ is
\[ \Delta_{Y_\bullet}(D_n) = \{ t_1 \cdot (D_n^2,0) + t_2 \cdot (0,1) | 0 \leq t_1, t_2, t_1 + t_2 \leq 1\}, \]
the convex hull of the points $(0,0)$, $(D_n^2,0)$ and $(0,1) \in \mathbb{R}^2$.
\end{thm}
\begin{proof}
If the nef divisor $D_n$ is not big then the constant $\mu(D_n,L)$ of Thm.~\ref{OB-surf-thm} vanishes. On the other hand, the negative part of the Zariski decomposition of $D_n$ is $0$, hence $\Delta_{Y_\bullet}(D_n)$ is the segment in $\mathbb{R}^2$ joining $(0,0)$ and $(0,1)$, by Thm.~\ref{OB-surf-thm}.

\noindent If the nef divisor $D_n$ is big, the set $\mathrm{Null}(D_n)$ introduced in section~\ref{OB-ZD-sec} is non-empty by Cor.~\ref{Null-cor}, and it consists of a finite number of irreducible and reduced curves $C_1, \ldots, C_K$ on $\widetilde{X}_n$ with negative-definite intersection matrix. The curve $C_j$ cannot be one of the exceptional divisors $E_i$ since then $D_n \cdot C_j = \epsilon_{\mathbb{P}^2}(L; x_1, \ldots, x_n) > 0$. Hence each $C_j$ is linearly equivalent to a divisor of the form $d_j \pi_n^\ast L - \sum_{i=1}^n m_{i,j} E_i$, with $d_j > 0, m_{i,j} \geq 0$, and we say that $C_j$ is of class $(d_j; m_{1,j}, \ldots, m_{n,j})$. Since $C_j^2 < 0$ there is at most one curve in $\mathrm{Null}(D_n)$ of a certain class.

\noindent Now, the set of classes of curves in $\mathrm{Null}(D_n)$ is invariant when the multiplicities $m_1, \ldots, m_n$ are permuted because such permutations are generated by elements of the Weyl group $W_n$, and we can apply the principle discussed in Sec.~\ref{curves-sec}. 

\noindent This implies that $\sum_{C \in \mathrm{Null}(D_n)} C$ is linearly equivalent to $a\pi_n^\ast L - b\cdot \sum_{i=1}^n E_i$, with $a, b > 0$. Furthermore, there are $k,t > 0$ such that
\[ D_n \equiv k \cdot \pi_n^\ast L + t \cdot \sum_{C \in \mathrm{Null}(D_n)} C :\]
$b, \epsilon_{\mathbb{P}^2}(L; x_1, \ldots, x_n) > 0$ imply $t > 0$. Since the intersection matrix of the $C_j$ is negative-definite we have $(\sum_{C \in \mathrm{Null}(D_n)} C)^2 < 0$. Consequently
\[ 0 = D_n \cdot  \sum_{C \in \mathrm{Null}(D_n)} C = ka + t \cdot ( \sum_{C \in \mathrm{Null}(D_n)} C)^2 \]
implies $k >0$.

\noindent Furthermore, $k = \mu_{\pi_n^\ast L}(D_n) := \max \{s | D_n - s \pi_n^\ast L\ \mathrm{pseudoeffective} \}$: Since $D_n - s \pi_n^\ast L \equiv (k-s) \pi_n^\ast L + t \cdot \sum_{C \in \mathrm{Null}(D_n)} C$ is pseudoeffective and $D_n$ is nef, we have
\[ 0 \leq D_n \cdot [ (k-s) \pi_n^\ast L + t \cdot \sum_{C \in \mathrm{Null}(D_n)} C ] = k - s, \]
so $s \leq k$. On the other hand, $D_n - k \pi_n^\ast L \equiv t \cdot \sum_{C \in \mathrm{Null}(D_n)} C$ is pseudoeffective.

\noindent Consequently, $D_n - s \pi_n^\ast L \in \overline{\Sigma_{D_n}}$ for all $0 \leq s \leq \mu_{\pi_n^\ast L}(D_n)$, by Lem.~\ref{clos-ZC-lem}. 

\noindent The form of $\Delta_{Y_\bullet}(D_n)$ then follows from Thm.~\ref{OB-surf-thm}: The $\alpha$-function is constantly $0$ if $x \not\in \bigcup_{C \in \mathrm{Null}(D_n)} \mathrm{Supp}(\pi_n(C))$, and the $\beta$-function is linear since the line given by $D_n - s \pi_n^\ast L$ always runs through the closure of the same Zariski chamber. Furthermore, the domain of the $\alpha$- and the $\beta$-function is $[0,k]$, and $\beta(0) = D_n \cdot \pi_n^\ast L = 1$. Since the area of $\Delta_{Y_\bullet}(D_n)$ is half of the volume of $D_n$ (\cite[Thm.A]{LM08}), that is $D_n^2/2$ as $D_n$ is nef (\cite[Cor.1.4.41 \& Def.2.2.31]{LazPAG1}), it follows that $k = D_n^2$.
\end{proof}

\noindent Using the facts on Seshadri constants collected in Sec.~\ref{curves-sec} we can finally calculate the Okounkov bodies $\Delta_{Y_\bullet}(L_{n,d,m})$, or at least predict their form.
\begin{thm}
Nagata's Conjecture holds for $n \geq 9$ if and only if for $d \geq \sqrt{n} \cdot m$ the Okounkov body $\Delta_{Y_\bullet}(L_{n,d,m})$ is the convex hull of $(0,0)$, $(d - \sqrt{n} \cdot m, 0)$, $(d - \sqrt{n} \cdot m, \sqrt{n} \cdot m)$ and $(0,d)$, that is a vertical strip in the triangle with corners $(0,0)$, $(d,0)$ and $(0,d)$.
\end{thm}
\begin{proof}
Nagata's Conjecture (see Conj.~\ref{Nag-conj}) predicts for $n \geq 9$ that 
\[ D_n^2 = (\pi_n^\ast L - \epsilon_n \sum_{i=1}^n E_i)^2 = 1 - n \cdot \frac{1}{\sqrt{n}^2} = 0. \]
Consequently, the Okounkov body $\Delta_{Y_\bullet}(D_n)$ is the line segment connecting $(0,0)$ and $(0,1)$, by Thm.~\ref{Sesh-OB-thm}. Rescaling and (\ref{resc-OB2-form}) yield the predicted form of $\Delta_{Y_\bullet}(L_{n,d,m})$.

\noindent Vice versa, if $\Delta_{Y_\bullet}(L_{n,d,m})$ is the vertical strip described in the statement, (\ref{rescale-form}) and rescaling show that $\Delta_{Y_\bullet}(\pi_n^\ast L - \frac{1}{\sqrt{n}}\sum_{i=1}^n E_i)$ is the line segment connecting $(0,0)$ and $(0,1)$. Then Thm.~\ref{Sesh-OB-thm} implies that $D_n = \pi_n^\ast L - \frac{1}{\sqrt{n}}\sum_{i=1}^n E_i$, that is 
\[ \epsilon_n = \frac{1}{\sqrt{n}}. \]
\end{proof}

\begin{thm}
The Okounkov body $\Delta_{Y_\bullet}(\pi_n^\ast L)$ is the triangle in $\mathbb{R}^2$ with corners $(0,0)$, $(1,0)$ and $(0,1)$.

\noindent For $1 \leq n \leq 8$ the Okounkov body $\Delta_{Y_\bullet}(\pi_n^\ast L - \frac{1}{3} \sum_{i=1}^n E_i)$ is a quadrilateral in $\mathbb{R}^2$ with corners $(0,0)$, $(\delta_n,0)$, $(\epsilon_n^\prime, 1 - \epsilon_n^\prime)$ and $(0,1)$, where $\epsilon_n^\prime := 1 - \frac{1}{3\epsilon_n}$ and $\delta_n := 1 - \frac{n\epsilon_n}{3}$.

\noindent The values of $\epsilon_n$, $\epsilon_n^\prime$ and $\delta_n$ for $n = 1, \ldots, 8$ are summarized in the following table:
\[ \begin{array}{c|cccccccc}
     n & 1 & 2 & 3 & 4 & 5 & 6 & 7 & 8 \\ \hline
     \epsilon_n & 1 & \frac{1}{2} & \frac{1}{2} & \frac{1}{2} & \frac{2}{5} & \frac{2}{5} & \frac{3}{8} & \frac{6}{17} \\
     \epsilon_n^\prime &  \frac{2}{3} & \frac{1}{3} & \frac{1}{3} & \frac{1}{3} & \frac{1}{6} & \frac{1}{6} & \frac{1}{9} & 
     \frac{1}{18} \\
     \delta_n & \frac{2}{3} & \frac{2}{3} & \frac{1}{2} & \frac{1}{3} & \frac{1}{3} & \frac{1}{5} & \frac{1}{8} & \frac{1}{17}  
    \end{array} \]

\noindent The iterative Okounkov body dissection associated to $\pi_1, \ldots, \pi_9$ and $\pi_9^\ast L - \frac{1}{3} \sum_{i=1}^9 E_i$ is given by the diagram below.

\noindent The Okounkov bodies $\Delta_{Y_\bullet}(L_{n,d,m})$, $n \leq 9$, are obtained from $\Delta_{Y_\bullet}(\pi_n^\ast L - \frac{1}{3} \sum_{i=1}^n E_i)$ using (\ref{resc-OB2-form}) and rescaling.
\end{thm}
\begin{proof}
$\Delta_{Y_\bullet}(\pi_n^\ast L)$ can be calculated directly from Thm.~\ref{OB-surf-thm}.

\noindent The statements on $\Delta_{Y_\bullet}(\pi_n^\ast L - \frac{1}{3} \sum_{i=1}^n E_i)$ and $\Delta_{Y_\bullet}(L_{n,d,m})$ for $1 \leq n \leq 9$ are direct consequences of Thm.~\ref{Sesh-OB-thm}, the list of Seshadri constants in Cor.~\ref{Sesh-cor} and the rescaling formula (\ref{resc-OB2-form}).
\end{proof}

\begin{center}
\begin{tikzpicture}
\draw (0,0) -- (10.2,0); 
\node [below] at (0.6,0) {$\frac{1}{17}$}; \node [below] at (1.27,0) {$\frac{1}{8}$}; \node [below] at (2.04,0) {$\frac{1}{5}$};
\node [below] at (3.4,0) {$\frac{1}{3}$};
\node [below] at (5.1,0) {$\frac{1}{2}$}; \node [below] at (6.8,0) {$\frac{2}{3}$}; \node [below] at (10.2,0) {$1$};

\draw (0,0) -- (0, 10.2);
\node [left] at (0,3.4) {$\frac{1}{3}$}; \node [left] at (0,6.8) {$\frac{2}{3}$}; \node [left] at (0,8.5) {$\frac{5}{6}$};
\node [left] at (0,9.07) {$\frac{8}{9}$}; \node [left] at (0,9.63) {$\frac{17}{18}$}; \node [left] at (0,10.2) {$1$};

\draw [dashed] (0,3.4) -- (6.8,3.4); \draw [dashed] (0,6.8) -- (3.4,6.8); \draw [dashed] (0,8.5) -- (1.7,8.5); 
\draw [dashed] (0,9.07) -- (1.13,9.07); \draw [dashed] (0,9.63) -- (0.57,9.63); 

\draw (0,10.2) -- (10.2,0);
\draw (0.57,9.63) -- (0.6,0); \draw (1.13,9.07) -- (1.27,0); \draw (1.7,8.5) -- (2.04,0); \draw (1.7,8.5) -- (3.4,0); \draw (3.4,6.8) -- (3.4,0); 
\draw (3.4,6.8) -- (5.1,0); \draw (3.4,6.8) -- (6.8,0); \draw (6.8,3.4) -- (6.8,0);
\end{tikzpicture}
\end{center}

\bibliographystyle{alpha}

\def\cprime{$'$}


\end{document}